\documentclass[a4paper,11pt]{amsart}

\usepackage[utf8]{inputenc}
\usepackage{lmodern}
\usepackage{url}
\usepackage{amsmath, amsthm, amssymb, amscd, accents, bm}
\usepackage{mathtools}
\usepackage{mdwlist}
\usepackage{paralist}
\usepackage{graphicx}
\usepackage{epstopdf}
\usepackage{datetime}
\usepackage{afterpage}
\usepackage[dvipsnames]{xcolor}
\usepackage{hyperref}
\usepackage[ruled,vlined]{algorithm2e}
\usepackage{caption}
\usepackage{subcaption}

\usepackage[foot]{amsaddr}

\usepackage{mathtools}

\newcommand\rurl[1]{%
  \href{https://#1}{\nolinkurl{#1}}%
}

\newtheorem{definition}{Definition}[section]
\newtheorem{theorem}[definition]{Theorem}
\newtheorem{lemma}[definition]{Lemma}
\newtheorem{corollary}[definition]{Corollary}
\newtheorem{proposition}[definition]{Proposition}
\newtheorem{problem}[definition]{Problem}
\newtheorem{remark}[definition]{Remark}

\newtheorem*{theorem*}{Theorem}
\newtheorem*{definition*}{Definition}

\def\N{{\mathbb N}}
\def\Z{{\mathbb Z}}
\def\R{{\mathbb R}}
\def\T{{\mathbb T}}
\def\C{{\mathbb C}}
\def\Q{{\mathbb Q}}

\newcommand{\lspan}{{\mathrm{span} }}

\newcommand{\lt}{{L^2(\R)}}

\newcommand{\s}{{\mathfrak{s}}}

\newcommand{\ift}{{\mathcal{F}^{-1}}}
\newcommand{\ft}{{\mathcal{F}}}

\makeatletter
\newcommand\thankssymb[1]{\textsuperscript{\@fnsymbol{#1}}}

\makeatletter
\def\@makefnmark{%
  \leavevmode
  \raise.9ex\hbox{\fontsize\sf@size\z@\normalfont\tiny\@thefnmark}}

\makeatletter
\def\bign#1{\mathclose{\hbox{$\left#1\vbox to8.5\p@{}\right.\n@space$}}\mathopen{}}

\usepackage{eqparbox}
\usepackage{makebox}
\newcommand\textoverset[3][]{\mathrel{\overset{\scriptsize{\eqmakebox[#1]{#2}}}{#3}}}

\begin{document}

\title[Discretization barriers in STFT phase retrieval]{On foundational discretization barriers in STFT phase retrieval}

\author[Philipp Grohs]{Philipp Grohs\thankssymb{1}\textsuperscript{,}\thankssymb{2}\textsuperscript{,}\thankssymb{3}}
\address{\thankssymb{1}Faculty of Mathematics, University of Vienna, Oskar-Morgenstern-Platz 1, 1090 Vienna, Austria}
\address{\thankssymb{2}Research Network DataScience@UniVie, University of Vienna, Kolingasse 14-16, 1090 Vienna, Austria}
\address{\thankssymb{3}Johann Radon Institute of Applied and Computational Mathematics, Austrian Academy of Sciences, Altenbergstrasse 69, 4040 Linz, Austria}
\email{philipp.grohs@univie.ac.at, philipp.grohs@oeaw.ac.at}
\author[Lukas Liehr]{Lukas Liehr\thankssymb{1}}
\email{lukas.liehr@univie.ac.at}

\date{\today}
\maketitle

\begin{abstract}
We prove that there exists \emph{no} window function $g \in \lt$ and \emph{no} lattice $\mathcal{L} \subset \R^2$ such that every $f \in \lt$ is determined up to a global phase by spectrogram samples $|V_gf(\mathcal{L})|$ where $V_gf$ denotes the short-time Fourier transform of $f$ with respect to $g$. Consequently, the forward operator
$$
f \mapsto |V_gf(\mathcal{L})|
$$
mapping a square-integrable function to its spectrogram samples on a lattice is never injective on the quotient space $\lt\bign/\sim$ with $f \sim h$ identifying two functions which agree up to a multiplicative constant of modulus one. We will further elaborate this result and point out that under mild conditions on the lattice $\mathcal{L}$, functions which produce identical spectrogram samples but do not agree up to a unimodular constant can be chosen to be real-valued. The derived results highlight that in the discretization of the STFT phase retrieval problem from lattice measurements, a prior restriction of the underlying signal space to a proper subspace of $\lt$ is inevitable.

\vspace{2em}

\end{abstract}

\section{Introduction}

The problem of recovering a function $f \in \lt$ from the absolute value of its short-time Fourier transform (STFT) has attracted a great deal of attention in recent years. This so-called \emph{STFT phase retrieval problem} has seen tremendous development of the corresponding theory and applications. It arises, for instance, in ptychography, which provides an attractive setting for the detailed understanding of the structure of materials \cite{Pfeiffer2018, RODENBURG200887}.
Starting from a window function $g \in \lt$ and a set $\mathcal{L} \subseteq \R^2$ of the time-frequency plane, one aims at recovering $f \in \lt$ from phaseless measurements of the form
$$
|V_gf(\mathcal{L})| = \left \{ |V_gf(x,\omega)| : (x,\omega) \in \mathcal{L} \right \}
$$
with $V_gf$ denoting the short-time Fourier transform of $f$ with respect to the window function $g$, defined by
\begin{equation}\label{def:stft}
    V_gf(x,\omega) = \int_\R f(t)\overline{g(t-x)}e^{-2\pi i \omega t} \, dt.
\end{equation}
The modulus of the STFT, $|V_gf|$, is called the spectrogram and it measures the distribution of the time-frequency content of $f$. Note that two functions $f,h$ which agree up to a global phase, i.e. there exists a constant $\nu \in \T \coloneqq \{ z \in \C : |z|=1 \}$ of modulus one such that $f = \nu h$, produce identical spectrograms. It follows that a reconstruction of $f $ from $|V_gf(\mathcal{L})|$ is only possible up to the ambiguity of a global phase factor. 
While the classical Fourier phase retrieval problem and Pauli problem suffer from non-trivial ambiguities or non-uniqueness, the usage of the STFT constitutes an attractive transform since suitable assumptions on the window function $g$ implies unique recovery (up to a global phase) from $|V_gf(\mathcal{L})|$, provided that $\mathcal{L}$ is a continuous domain such as $\mathcal{L}=\R^2$ \cite[Section 4]{GrohsKoppensteinerRathmair}. In applications, however, the spectrogram is only accessible at a discrete set, most notably at samples on a lattice. In this article, we address the natural question of the uniqueness of the STFT phase retrieval problem in the situation where $\mathcal{L}$ is a lattice in the time-frequency plane. We prove that \emph{\begin{center} no window function $g \in \lt$ and no lattice $\mathcal{L} \subset \R^2$ \\ achieves unique recovery (up to a global phase) of signals in $\lt$.\end{center}}
\noindent This statement reveals a fundamental barrier in the discretization of the STFT phase retrieval problem from lattice measurements: the prior restriction to a proper subspace of $\lt$ in the discretization of the STFT phase retrieval problem from lattice samples is inevitable.

\subsection{Main results}
We now state the main results of the paper in a mathematically precise form. To that end, we introduce the equivalence relation $f \sim h$ which indicates that $f$ and $h$ equal up to a global phase, i.e.
$$
f \sim h \iff \exists \nu \in \T : f=\nu h.
$$

\begin{definition}\label{definition:uniqueness_pair}
Let $g \in \lt$ and $\mathcal{L} \subseteq \R^2$. We say that $(g,\mathcal{L})$ is a uniqueness pair of the STFT phase retrieval problem if every $f \in \lt$ is determined up to a global phase by $|V_gf(\mathcal{L})|$, i.e. the implication
$$
|V_gf(\mathcal{L})| = |V_gh(\mathcal{L})| \implies f \sim h
$$
holds true for every $f,h \in \lt$.
\end{definition}

In a more abstract language, the property of a uniqueness pair may be replaced by the demand on the forward operator $f \mapsto |V_gf(\mathcal{L})|$ being injective on the the quotient space $\lt\bign/\sim.$ Recall that a lattice is a set $\mathcal{L}=L\Z^2$ with $L \in \mathrm{GL}_2(\R)$ being an invertible matrix, the generating matrix of $\mathcal{L}$. A shifted lattice is a set $\mathcal{S} \subset \R^2$ of the form $\mathcal{S} = z + \mathcal{L}$ where $\mathcal{L}$ is a lattice and $z \in \R^2$ is a vector.
In \cite[Theorem 1 and Remark 4]{alaifari2020phase} the authors show that $(\varphi,\mathcal{L})$ is not a uniqueness pair provided that $\varphi(t)=e^{-\pi t^2}$ is a centered Gaussian and $\mathcal{L}$ is a shifted lattice. The following theorem significantly generalizes this result (see Section \ref{sec:main}).

\begin{theorem}\label{thm:mmm}
Suppose that $g \in \lt$ is an arbitrary window function and let $\mathcal{S} \subset \R^2$ be a subset of the time-frequency plane. Then $(g,\mathcal{S})$ is never a uniqueness pair, provided that $\mathcal{S}$ is a shifted lattice, i.e. $\mathcal{S} = z+ L\Z^2$ for some $L \in \mathrm{GL}_2(\R)$ and some $z \in \R^2$.
\end{theorem}

The previous statement is in stark contrast to the case where phase information is available: it is known that for every lattice $\mathcal{L}=L\Z^2, L \in \mathrm{GL}_2(\R)$, which satisfies $|\det(L)|\leq 1$ there exists a window function $g \in \lt$ such that every $f \in \lt$ is uniquely determined by $V_gf(\mathcal{L})$ \cite[Theorem 11]{Heil2007}. In addition, mild conditions on a window function $g$ imply that the Gabor system $\{ e^{2\pi i \ell \cdot } g(\cdot - k ) : (k,\ell) \in a\Z \times b \Z \}$ forms a frame for $\lt$ (which is a much stronger statement than being complete) provided that $a,b > 0$ are sufficiently small, see \cite[Section 3.5]{Heil2007} and the references therein. In fact, Theorem \ref{thm:mmm} above follows from a more general observation that shifted parallel lines do not guarantee uniqueness. We say that $\mathcal{P} \subset \R^2$ is a set of shifted parallel lines if $\mathcal{P}$ arises from a set of the form $\R \times h\Z, h>0,$ via a rotation around the origin, followed by a translation.

\begin{theorem}\label{thm:2}
Suppose that $g \in \lt$ is an arbitrary window function and let $\mathcal{P} \subset \R^2$. Then $(g,\mathcal{P})$ is never a uniqueness pair, provided that $\mathcal{P}$ is a set of shifted parallel lines in the time-frequency plane.
\end{theorem}

\begin{figure}[b]
\centering
   \includegraphics[width=12.5cm]{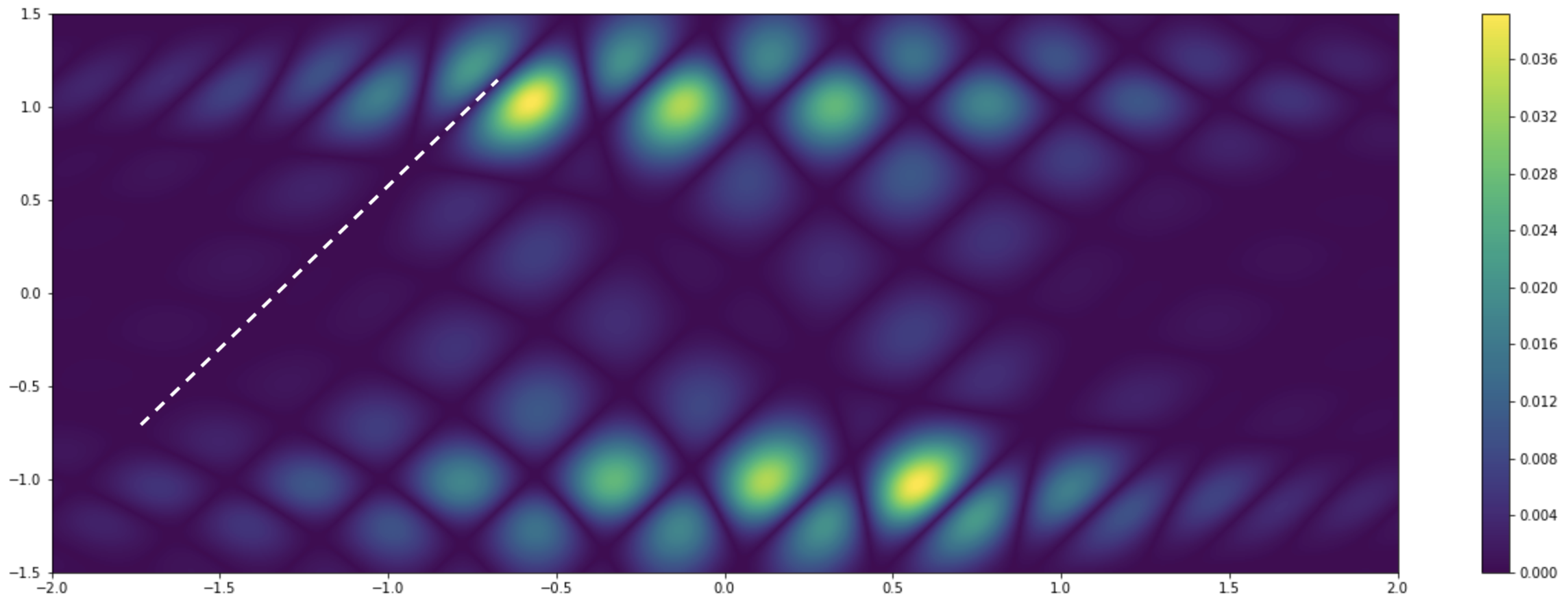}
\caption{Plot of the pointwise distance $Q$ of the square of two spectrograms, $Q \coloneqq ||V_gf_1|^2-|V_gf_2|^2|$, on the rectangle $[-2,2]\times[-1.5,1.5]$ with window function $g(t)=e^{-|t|}$. The functions $f_1$ and $f_2$ are chosen in such a way that $Q$ vanishes on parallel lines in the time-frequency plane, one of which is highlighted by the white dotted line. Such functions $f_1$ and $f_2$ can be obtained for every window function $g$ and they have the additional property that they do not agree up to a global phase. Their precise construction is the content of Section \ref{sec:main}.}
\label{fig:intro_plot}
\end{figure}

The conclusion of Theorem \ref{thm:2} is visualized in Figure \ref{fig:intro_plot}.
Prominent choices of window functions in Gabor analysis and its applications have the property that they are real-valued (Gaussians, Hermite functions, Airy disc functions, Hanning windows, rectangular windows). One could inquire about the question of whether a real-valuedness assumption on the underlying signal space achieves uniqueness from lattice samples, i.e. one specializes the considered input functions to belong to the space
$$
L^2(\R,\R) = \left \{ f \in \lt : f \ \text{is real-valued} \right \}.
$$
Under mild conditions on a lattice $\mathcal{L}$, the STFT phase retrieval problem still fails to be unique.

\begin{theorem}
Let $g \in L^2(\R,\R)$ be a real-valued window function and $\mathcal{L}=L\Z^2$ a lattice generated by 
\begin{equation*}
    L = \begin{pmatrix}
a & b\\
c & d
\end{pmatrix} \in \mathrm{GL}_2(\R) .
\end{equation*}
If $a$ and $b$ as well as $c$ and $d$ are linearly dependent over $\Q$ then there exist two real-valued functions $f_1,f_2 \in L^2(\R,\R)$ which do not agree up to a global phase, $f_1 \nsim f_1$, but their spectrograms agree on $\mathcal{L}$,
$$
|V_gf_1(\mathcal{L})| = |V_gf_2(\mathcal{L})|.
$$
\end{theorem}

In particular, the conclusion of the previous theorem holds true if the generating matrix $L$ has rational entries, $L \in \mathrm{GL}_2(\Q),$ or if the lattice is rectangular, i.e. $\mathcal{L}=\alpha \Z \times \beta \Z$ for some $\alpha,\beta \in \R \setminus \{ 0\}$.

\subsection{Outline}
We quickly outline the structure of the paper. In Section \ref{sec:preliminaries} we define frequently used operators and their interrelations. Besides, we present the necessary background on the fractional Fourier transform and shift-invariant spaces. Based on this discussion we proceed in proving the main results of the article in Section \ref{sec:main}. We conclude the article with Section \ref{sec:conclusion}, where we compare our results to previous work and outline a list of open problems suggested by the reported results.

\section{Preliminaries on fractional Fourier transform and associated shift-invariant spaces}\label{sec:preliminaries}

In this section, we collect several preliminary definitions and results about operators arising in time-frequency analysis, the fractional Fourier transform, and shift-invariant spaces.

\subsection{Basic operations}\label{sec:basic_operations}

The short-time Fourier transform of a function $f \in \lt$ with respect to a window function $g \in \lt$ is the map
$$
V_gf(x,\omega) = \int_\R f(t)\overline{g(t-x)}e^{-2\pi i \omega t} \, dt,
$$
which is defined on the time-frequency plane $\R^2$. This is a uniformly continuous map which measures the contribution to $f(t)$ of the frequency $\omega$ near $t=x$. In particular, it is pointwise defined. Hence, the STFT phase retrieval problem from lattice samples is well-defined.
The so-called spectrogram of $f$ with respect to $g$ is the modulus of the STFT, i.e. the map
$$
\R^2 \to \R_{\geq 0}, \ \ (x,\omega) \mapsto |V_gf(x,\omega)|.
$$
Note that $V_gf(x,\omega)$ is the Fourier transform of $f\overline{g(\cdot -x)}$ evaluated at $\omega$. The Fourier transform of a map $u \in L^1(\R)$ is given by
$$
\ft u(\omega) \coloneqq \hat u(\omega) \coloneqq \int_\R u(t)e^{-2 \pi i \omega t} \, dt.
$$
Recall that $\ft u$ is an element of $C_0(\R)$, the Banach space of all continuous functions on the real line vanishing at infinity. The inverse Fourier transform is the map $\ift u(t)=\ft u(-t)$. The operator $\ft$ extends from $L^1(\R) \cap \lt$ to a unitary operator mapping $\lt$ bijectively onto $\lt$. In addition to the Fourier operator, we introduce some key-operations in time-frequency analysis. For $\tau,\nu \in \R$, the translations operator $T_\tau : \lt \to \lt$, the modulation operator $M_\nu : \lt \to \lt$ and the reflection operator $\mathcal{R} : \lt \to \lt$ are defined via
\begin{equation*}
    \begin{split}
        T_\tau f(t) &= f(t-\tau), \\
        M_\nu f(t) &= e^{2\pi i \nu t} f(t), \\
        \mathcal{R}f(t) &= f(-t).
    \end{split}
\end{equation*}
Using translation and modulation, the short-time Fourier transform of $f$ with respect to $g$ simplifies to the expression $V_gf(x,\omega) = \langle f, M_\omega T_x g \rangle$ where $\langle a,b \rangle = \int_\R a(t)\overline{b(t)} \, dt$ denotes the $L^2$-inner product of two functions $a,b \in \lt$. The next Lemma summarizes elementary relations between the operators $\ft, V_g, T_\tau, M_\nu$ and $\mathcal{R}$ which will be used throughout the article. The simple proofs can be found, for instance, in \cite[Chapter 1--3]{Groechenig}.

\begin{lemma}\label{lemma:comm_relations}
For every $\tau,\nu,x,\omega \in \R$ and every $f,g \in \lt$, the operators defined above satisfy the relations
\begin{enumerate}
    \item $T_\tau M_\nu = e^{-2\pi i \tau \nu} M_\nu T_\tau$
    \item $\ft T_\tau = M_{-\tau} \ft$
    \item $\ift T_\tau = M_\tau \ift$
    \item $T_\tau \mathcal{R} = \mathcal{R} T_{-\tau}$
    \item $\overline{\ft f} = \mathcal{R} \ft \overline{f}$.
\end{enumerate}
Moreover, the STFT satisfies the covariance property
\begin{enumerate}
    \item[(6)] $V_g(T_\tau M_\nu f)(x,\omega) = e^{-2\pi i \tau \omega} V_gf(x-\tau, \omega -\nu)$.
\end{enumerate}
\end{lemma}

\subsection{Fractional Fourier transform}\label{sec:fracft}

Denote by $\{ H_n \}_{n \in \N_0}$ the Hermite basis functions on $\R$, defined by
$$
H_n(t) = \frac{2^{1/4}}{\sqrt{n!}} \left( -\frac{1}{\sqrt{2\pi}} \right)^n e^{\pi t^2} \left( \frac{d}{dt} \right)^n e^{-2\pi t^2}.
$$
The system $\{ H_n \}_{n \in \N_0}$ constitutes and orthonormal basis for $\lt$.
For $\theta \in \R$, the unitary operator $\ft_\theta : \lt \to \lt$,
$$
\ft_\theta f = \sum_{n=0}^\infty e^{-i\theta n} \langle f,H_n \rangle H_n,
$$
is called the fractional Fourier transform of order $\theta$. A detailed investigation of this transform and its appearance in quantum mechanics can be found in the early paper by Namias \cite{Namias}. Besides, Namias' article provides an alternative and frequently used definition of the operator $\ft_\theta$ as an integral transform invoking chirp modulations \cite[Section 3]{Namias}. The following properties of the fractional Fourier transform are used throughout the present article \cite[Section 4]{Namias}.

\begin{lemma}\label{lemma:fracft_properties}
For every $\theta \in \R$ and every $\xi \in \R$ the fractional Fourier transform $\ft_\theta$ has the following properties:
\begin{enumerate}
    \item $\ft_0 = \mathrm{Id}, \ft_\frac{\pi}{2} = \ft, \ft_{-\frac{\pi}{2}} = \ft^{-1}$ and $\ft_\pi = \mathcal{R}$
    \item $\ft_{\theta+\xi} = \ft_\theta \ft_\xi$
    \item $\ft_\theta$ commutes with $\mathcal{R}$
    \item If $f \in \lt$ then $\overline{\ft_\theta f} = \ft_{-\theta} \overline{f}$
\end{enumerate}
\end{lemma}

In the context of time-frequency analysis, the perhaps most significant property of the operator $\ft_\theta$ is its relation to rotations in the time-frequency plane \cite{Jaming,Lohmann93}. To present the precise mathematical formulation of this fact, we define for $\theta \in \R$ the rotation matrix $R_\theta \in \R^{2 \times 2}$ via
\begin{equation}\label{def:rotation_matrix}
    R_\theta = \begin{pmatrix}
\cos \theta & -\sin \theta\\
\sin\theta & \cos\theta
\end{pmatrix}.
\end{equation}
The rotation property of the fractional Fourier transform reads as follows.

\begin{theorem}\label{thm:frac_rotation}
If $R_\theta$ denotes the rotation matrix as given in equation \eqref{def:rotation_matrix} then for every $f,g\in \lt$ and every $(x,\omega) \in \R^2$ one has
    $$
    V_gf(R_\theta(x,\omega)) = e^{-\pi i q(\theta)} e^{\pi i x \omega} V_{\ft_\theta g}(\ft_\theta f)(x,\omega)
    $$
    with $q(\theta) = (x\cos \theta - \omega \sin \theta)(x\sin \theta + \omega \cos \theta)$.
\end{theorem}

A proof of Theorem \ref{thm:frac_rotation} can be found in \cite{Almeida}. Moreover, the reader may consult \cite[Section 3.4]{Jaming} for a discussion of the rotation property. The phase retrievability of functions $f \in \lt$ from measurements of the form $\{ |\ft_\theta f| : \theta \in \Theta \}$, with $\Theta \subset \R$ constituting a set of orders, was studied by Jaming \cite{Jaming} and Carmeli et al. \cite{CARMELI}.

\subsection{Shift-invariant spaces}\label{sec:si_spaces}

In this section, we shall discuss shift-invariant spaces, both in the classical sense as well as generalizations in the context of fractional Fourier transform. We start by introducing classical shift-invariant spaces. Consider a generating function $u \in \lt$ and a constant $\s>0$. The subspace $\mathcal{V}_\s(u) \subset \lt$, defined as the $L^2$-closure of the $\s\Z$-shifts of $u$,
\begin{equation}\label{definition:calssical_si_space}
    \mathcal{V}_\mathfrak{s}(u) \coloneqq \overline{\lspan \left \{ T_{\s k}u : k \in \Z \right \}},
\end{equation}
is called the (principle) shift-invariant space generated by $u$ and step-size $\s>0$. Shift-invariance means that whenever $h \in \mathcal{V}_\s(u)$, so is $T_{\s n} h$ for every $n \in \Z$. A characterization of functions belonging to $\mathcal{V}_\s(u)$ was given in \cite{DEBOOR199437}. 
If $h \in \mathcal{V}_\s(u)$ is such that $h$ has a representation of the form $h = \sum_{k \in \Z} c_k T_{\s k} u$ for some sequence $\{ c_k \} \subset \C$ then we call $\{ c_k \}$ a defining sequence of $h$.
If $\{ c_k \}$ belongs to the space $c_{00}(\Z)$ of sequences with only finitely many nonzero components then clearly $\sum_{k \in \Z} c_k T_{\s k}u$ is a well-defined function in $\mathcal{V}_\s(u)$ and no convergence issues appear. Care should be taken if $\{ c_k \} \notin c_{00}(\Z)$: since we did not make any additional assumptions on the generating function $u \in \lt$ (in particular, we are not assuming that $\{ T_{\s k} u : k\in \Z \}$ is a Riesz basis or frame for $\mathcal{V}_\s(u)$), we may not be able to represent every function $f \in \mathcal{V}_\s(u)$ in the form $\sum_{k \in \Z} c_k T_{\s k} u$. An assumption which guarantees unconditional convergence for every square-summable sequence $\{ c_k \} \in \ell^2(\Z)$ is the requirement that $\{ T_{\s k} u : k\in \Z \}$ is a Bessel sequence \cite[Corollary 3.2.5]{christensenBook}. Recall that a sequence $\{ x_k \}$ in a Hilbert space $H$ is a Bessel sequence if 
$$
\forall x \in H : \sum_{k \in \Z} |\langle x,x_k \rangle|^2 < \infty.
$$
\begin{theorem}\label{thm:bessel}
If $\{ x_k : k \in \Z \}$ is a Bessel sequence in a Hilbert space $H$ then for every $\{ c_k \} \in \ell^2(\Z)$ the series $\sum_{k \in \Z} c_k x_k$ converges unconditionally in $H$.
\end{theorem}
If one would apply the Fourier transform to $\mathcal{V}_\s(u)$ then the resulting space
$$
\mathcal{M}_\s(u) = \left \{ \ft h : h \in \mathcal{V}_\s(u) \right \}
$$
is modulation invariant (or: invariant under a Fourier shift), i.e. $M_{\s n}h \in \mathcal{M}_\s(u)$ for every $n \in \Z$ provided that $h \in \mathcal{M}_\s(u)$. Now observe that
$$
T_{\s k} = \mathrm{Id} \, T_{\s k} \,  \mathrm{Id}, \ \ M_{\s k} = \ft^{-1} T_{\s k} \ft,
$$
or, in the language of fractional Fourier transform,
$$
T_{\s k} = \ft_0 T_{\s k} \ft_0, \ \ M_{\s k} = \ft_{-\frac{\pi}{2}} T_{\s k} \ft_{\frac{\pi}{2}}.
$$
The previous identities motivate the definition of a fractional Fourier shift.

\begin{definition}
Let $\tau,\theta \in \R$. The operator $T_\tau^\theta$, defined by
$$
T_\tau^\theta \coloneqq \ft_{-\theta} T_{\tau} \ft_\theta
$$
is called the fractional Fourier shift by $\tau$ of order $\theta$.
\end{definition}

We have $T_\tau^0 = T_\tau$ and $T_\tau^{\frac{\pi}{2}} = M_\tau$. Hence, the operator $T_\tau^\theta$ interpolates between the ordinary shift and the modulation, i.e. a Fourier shift. In a natural way, the fractional Fourier shift extends the concept of classical shift-invariant spaces.

\begin{definition}
Let $u \in \lt$ be a generating function and $\s >0$. The space 
$$
\mathcal{V}_\mathfrak{s}^\theta (u) \coloneqq \overline{\lspan \left \{ T_{\s k}^\theta u : k \in \Z \right \}}
$$
is called the shift-invariant space associated to the fractional Fourier shift generated by $u$ and step-size $\s >0$. In the case $\theta = 0$ we have $\mathcal{V}_\s^0(u) = \mathcal{V}_\s(u)$ with $\mathcal{V}_\s(u)$ being defined as in \eqref{definition:calssical_si_space}.
\end{definition}

For details and properties on shift-invariant spaces associated to the fractional Fourier shift, the reader may consult the article \cite{Zayed} and the references therein. A direct application of Theorem \ref{thm:bessel} gives the following statement.

\begin{corollary}\label{cor:uncond_convergence}
Let $u \in \lt$ be a generating function, $\s >0$ a step-size and $\theta \in \R$ an order. If the system of fractional Fourier shifts $\{ T_{\s k}^\theta u : k \in \Z \}$ constitutes a Bessel sequence in $\lt$ then for every $\{ c_k \} \in \ell^2(\Z)$ the series $\sum_{k \in \Z} c_k T_{\s k}^\theta u$ converges unconditionally to an element in $\mathcal{V}_\s^\theta(u)$.
\end{corollary}

\section{Main results}\label{sec:main}

In this section, we present and prove the main results of the paper. We derive functions $f_1,f_2 \in \lt$ which have the property that their spectrograms agree on certain selected subsets of the time-frequency plane. Moreover, we discuss under which assumptions such functions do not agree up to a global phase and which conditions imply that they can be chosen to be real-valued.

\subsection{Identical spectrograms and non-equivalence conditions}\label{subsec:general_case}

Let $u \in \lt, \s>0$ and $\theta \in \R$. In the sequel we shall fix the following notation: whenever $f \in \mathcal{V}_\s^\theta(u)$ has a representation of the form $f = \sum_{k \in \Z} c_k T_{\s k}^\theta u$ with a defining sequence $\{ c_k \} \subset \C$ then the map $f_\times$ is defined by
$$
f_\times \coloneqq \sum_{k \in \Z} \overline{c_k} T_{\s k}^\theta u,
$$
i.e. $f_\times$ arises from $f$ via complex conjugation of the defining sequence $\{ c_k \}$. We now show that a suitable choice of the generator $u$ implies equality of the spectrogram of $f$ and $f_\times$ on parallel lines of the form
\begin{equation}
    R_\theta (\R \times h\Z) \coloneqq \left \{ 
    \begin{pmatrix}
\cos \theta & -\sin \theta\\
\sin\theta & \cos\theta
\end{pmatrix}
\begin{pmatrix}
x\\
\omega
\end{pmatrix}
: x \in \R, \, \omega \in h\Z
  \right \}, \ \ h > 0.
\end{equation}

\begin{theorem}\label{theorem:parallel_lines}
Let $g \in \lt$ be a window function, $\s>0$ a step-size and $\theta \in \R$ an order. If $f \in \mathcal{V}_\s^\theta(\mathcal{R}g)$ has defining sequence $\{ c_k \} \in c_{00}(\Z)$ then 
$$
|V_gf(R_\theta(\R \times \tfrac{1}{\s} \Z))| = |V_g(f_\times)(R_\theta(\R \times \tfrac{1}{\s} \Z))|.
$$
If, in addition, the system $\{ T_{\s k}^\theta(\mathcal{R}g) : k \in \Z \}$ is a Bessel sequence then the same conclusion holds true, provided that $f$ has defining sequence $\{ c_k \} \in \ell^2(\Z)$.
\end{theorem}
\begin{proof}
Let $(x,\omega) \in \R^2$ and let $R_\theta$ be the rotation matrix as defined in equation \eqref{def:rotation_matrix}. Suppose that $f \in \mathcal{V}_\s^\theta(\mathcal{R}g)$ has defining sequence $\{c_k \} \in c_{00}(\Z)$. Invoking the definition of the fractional Fourier shift, $f$ can be written as
$$
f = \ft_{-\theta} \sum_{k \in \Z} c_k T_{\s k}\ft_\theta \mathcal{R}g.
$$
We now evaluate the spectrogram of $f$ with respect to the window function $g$ at a $\theta$-rotation of $(x,\omega)$. 
Consulting the properties of the fractional Fourier transform as presented in Lemma \ref{lemma:fracft_properties} yields

\begin{align}\label{33}
        |V_gf(R_\theta(x,\omega))|  & \textoverset[0]{Theorem \ref{thm:frac_rotation}}{=} \left | V_{\ft_\theta g} \left ( \sum_{k\in\Z} c_k T_{\s k}\ft_\theta \mathcal{R}g \right ) (x,\omega) \right | \nonumber \\
          & \textoverset[0]{linearity}{=} \left | \sum_{k \in \Z} c_k V_{\ft_\theta g}(T_{\s k}\ft_\theta \mathcal{R}g)(x,\omega) \right | \nonumber \\
          & \textoverset[0]{Lemma \ref{lemma:comm_relations}(6)}{=} \left | \sum_{k \in \Z} c_k e^{-2 \pi i \s k \omega} V_{\ft_\theta g}(\ft_\theta \mathcal{R}g)(x-\s k,\omega) \right | \nonumber \\
         & \textoverset[0]{}{=} \left | \sum_{k \in \Z} c_k e^{-2 \pi i \s k \omega} \ft \left ( (\ft_\theta \mathcal{R}g)( \overline{T_{x-\s k} \ft_\theta g})  \right )(\omega)\right | \nonumber \\
         & \textoverset[0]{Lemma \ref{lemma:fracft_properties}(4)}{=} \left | \sum_{k \in \Z} c_k e^{-2 \pi i \s k \omega} \ft \left ( (\ft_\theta \mathcal{R}g) (T_{x-\s k} \ft_{-\theta} \overline{g})  \right )(\omega)\right |.
\end{align}
We continue by inspecting the spectrogram of $f_\times$.
Replacing $c_k$ with $\overline{c_k}$ in equation \eqref{33} shows that the spectrogram of $f_\times$ satisfies
\begin{align*}
        |V_g(f_\times)(R_\theta(x,\omega))| & \textoverset[0]{}{=} \left | \sum_{k \in \Z} \overline{c_k} e^{-2 \pi i \s k \omega} \ft \left ( (\ft_\theta \mathcal{R}g) (T_{x-\s k} \ft_{-\theta} \overline{g})  \right )(\omega) \right | \\
        & \textoverset[0]{}{=} \left | \sum_{k \in \Z} c_k e^{2 \pi i \s k \omega} \overline{\ft \left ( (\ft_\theta \mathcal{R}g) (T_{x-\s k} \ft_{-\theta} \overline{g})  \right )(\omega)} \right | \\
        & \textoverset[0]{Lemma \ref{lemma:comm_relations}(5)}{=} \left | \sum_{k \in \Z} c_k e^{2 \pi i \s k \omega} \mathcal{R} \ft \left ( \overline{(\ft_\theta \mathcal{R}g) (T_{x-\s k} \ft_{-\theta} \overline{g})  }\right )(\omega) \right | \\
        & \textoverset[0]{Lemma \ref{lemma:fracft_properties}(4)}{=} \left | \sum_{k \in \Z} c_k e^{2 \pi i \s k \omega} \mathcal{R} \ft \left ( (\ft_{-\theta} \mathcal{R}\overline{g}) (T_{x-\s k} \ft_{\theta} g)  \right )(\omega) \right | \\
        & \textoverset[0]{Lemma \ref{lemma:fracft_properties}(3)}{=} \left | \sum_{k \in \Z} c_k e^{2 \pi i \s k \omega}  \ft \left ( \mathcal{R}[(\ft_{-\theta} \mathcal{R}\overline{g}) (T_{x-\s k} \ft_{\theta} g)]  \right )(\omega) \right | \\
        & \textoverset[0]{Lemma \ref{lemma:fracft_properties}(3)}{=} \left | \sum_{k \in \Z} c_k e^{2 \pi i \s k \omega}  \ft \left ( (\ft_{-\theta} \overline{g}) (T_{\s k-x} \ft_{\theta} \mathcal{R} g)  \right )(\omega) \right | \\
        & \textoverset[0]{Lemma \ref{lemma:comm_relations}(2)}{=} \left | \sum_{k \in \Z} c_k e^{2 \pi i \s k \omega} e^{-2\pi i \omega (\s k -x)}  \ft \left ( (T_{x-\s k}\ft_{-\theta} \overline{g}) (\ft_{\theta} \mathcal{R} g)  \right )(\omega) \right | \\
        & \textoverset[0]{}{=} \left |  \sum_{k \in \Z} c_k \ft \left ( (T_{x-\s k}\ft_{-\theta} \overline{g}) (\ft_{\theta} \mathcal{R} g)  \right )(\omega) \right |.
\end{align*}
Now let $\omega = \frac{n}{\s}$ for some $n \in \Z$. In this case, the term $e^{-2 \pi i \s k \omega}$ appearing in the equation \eqref{33} simplifies to $e^{-2 \pi i \s k \omega} = e^{-2 \pi i kn}=1$. This shows that $|V_gf(R_\theta(x,\tfrac{n}{\s}))| = |V_g(f_\times)(R_\theta(x,\tfrac{n}{\s}))|$ for every $n \in \Z$. Since $x \in \R$ was arbitrary, we obtain the desired equality of the two spectrograms on the parallel lines $R_\theta(\R \times \tfrac{1}{\s} \Z)$.
Under the additional assumption that the system $\{ T_{\s k}^\theta(\mathcal{R}g) : k \in \Z \}$ forms a Bessel sequence, Corollary \ref{cor:uncond_convergence} implies that for every $\{ c_k \} \in \ell^2(\Z)$ the series $\sum_{k \in \Z} c_k T_{\s k} \mathcal{R}g$ converges unconditionally in $\lt$. Since $\{ \overline{c_k} \} \in \ell^2(\Z)$ if and only if $\{c_k \} \in \ell^2(\Z)$, the same holds for the series $\sum_{k \in \Z} \overline{c_k} T_{\s k} \mathcal{R}g$. By continuity and linearity of the fractional Fourier transform, the exact same argument as above implies equality of the spectrogram of $f$ and $f_\times$ on $R_\theta(\R \times \tfrac{1}{\s} \Z)$.
\end{proof}

The statement obtained in Theorem \ref{theorem:parallel_lines} is visualized in Figure \ref{fig:hermite} and Figure \ref{fig:exp_arctan}.
The functions $f$ and $f_\times$ defined in Theorem \ref{theorem:parallel_lines} constitute candidates of functions for which their spectrograms agree on parallel lines whereas the functions themselves might not agree up to a global phase. We now examine criteria on the defining sequence $\{ c_k \}$ of $f$ so that $f$ and $f_\times$ do not agree up to a global phase. Before doing so, we recall the concept of $\omega$-independence \cite[p. 34]{Young}.

\begin{definition}
A sequence $\{ x_n\}_{n \in \Z}$ of elements in a Banach space $X$ is said to be $\omega$-independent if the equality
\begin{equation}\label{norm_conv_series}
    \sum_{n \in \Z} c_n x_n = 0
\end{equation}
is possible only for $c_n = 0, n \in \Z$.
\end{definition}

Recall that in the definition of a linearly independent system, the norm-convergent series in \eqref{norm_conv_series} is replaced by a finite sum. Systems of translates of a function which constitute an $\omega$-independent sequence were studied in \cite{rrron}. We continue by introducing a class of sequences which obeys a geometric property.

\begin{definition}
The subset $\ell^2_\mathcal{O}(\Z) \subset \ell^2(\Z)$ of square-summable sequences whose elements do not lie on a line in the complex plane passing through the origin is defined as
$$
\ell^2_\mathcal{O}(\Z) \coloneqq \left \{  \{c _k \} \in \ell^2(\Z): \nexists \, \alpha \in \R \ \mathrm{s.t.} \ \{ c_k \} \subset e^{i\alpha} \R   \right \}.
$$
\end{definition}

Note that $\ell^2_\mathcal{O}(\Z)$ is invariant under complex conjugation, i.e. $\{ c_k \} \in \ell^2_\mathcal{O}(\Z)$ if and only if $\{ \overline{c_k} \} \in \ell^2_\mathcal{O}(\Z)$. In order to show that sequences in $\ell^2_\mathcal{O}(\Z)$ produce functions $f$ and $f_\times$ which do not agree up to a global phase, we make use of the following simple observation.

\begin{lemma}\label{lemma:funcion_values_lines}
Let $Y$ be an arbitrary set and $f : Y \to \C$ a complex-valued map. Then $f \sim \overline{f}$ if and only if there exists an $\alpha \in \R$ such that $f(y) \in e^{i\alpha}\R$ for every $y \in Y$.
\end{lemma}
\begin{proof}
We first prove the sufficiency of the property $f \sim \overline{f}$. To this end, let $\nu \in \R$ such that $f(y)=e^{i\nu}\overline{f(y)}$ for every $y \in Y$. Let $f(y)=e^{i\phi(y)}|f(y)|$ and $\overline{f(y)}=e^{-i\phi(y)}|f(y)|$ with $\phi$ being a phase function corresponding to $f$. Let $\Lambda = \{ y \in Y : f(y) \neq 0 \}$. It follows that
$$
e^{i(2\phi(y)-\nu)} = 1
$$
for every $y \in \Lambda$ which implies that there exists an $n(y) \in \Z$ so that $\phi(y) = \pi n(y) + \tfrac{\nu}{2}$. Thus,
$$
f(y) = e^{i\pi n(y)}e^{i\frac{\nu}{2}} |f(y)| = e^{i\frac{\nu}{2}} \left ( (-1)^{n(y)} |f(y)| \right )
$$
which shows that $f$ takes values on the line $e^{i\frac{\nu}{2}}\R$. Conversely, suppose that $f$ takes values on a line $e^{i\alpha}\R$. Then $f = e^{i\alpha}|f| = e^{2i\alpha} \overline{f}$ which implies that $f$ and $\overline{f}$ agree up to a global phase.
\end{proof}

With the aid of the previous Lemma we can now show that the assumption on a defining sequence belonging to $\ell^2_\mathcal{O}(\Z)$ gives rise to non-equivalent function pairs. This is the main content of

\begin{theorem}\label{theorem:non_uniqueness}
Let $0 \neq u \in \lt, \s >0$ and suppose that $f \in \mathcal{V}_\s^\theta(u)$ has defining sequence $\{ c_k \} \subset \C$. Then the following holds:
\begin{enumerate}
    \item If $\{ c_k \} \in \ell^2_\mathcal{O}(\Z) \cap c_{00}(\Z)$ then $f \nsim f_\times$.
    \item If the generating function $u$ has the property that the system of fractional Fourier translates $ \{ T_{\s k}^\theta u : k \in \Z \}$ forms an $\omega$-independent Bessel sequence then
    $$
    \{ c_k \} \in \ell^2_\mathcal{O}(\Z) \implies f \nsim f_\times.
    $$
\end{enumerate}
\end{theorem}

\begin{figure}
\centering
\begin{subfigure}{.5\textwidth}
  \centering
  \includegraphics[width=0.9\linewidth]{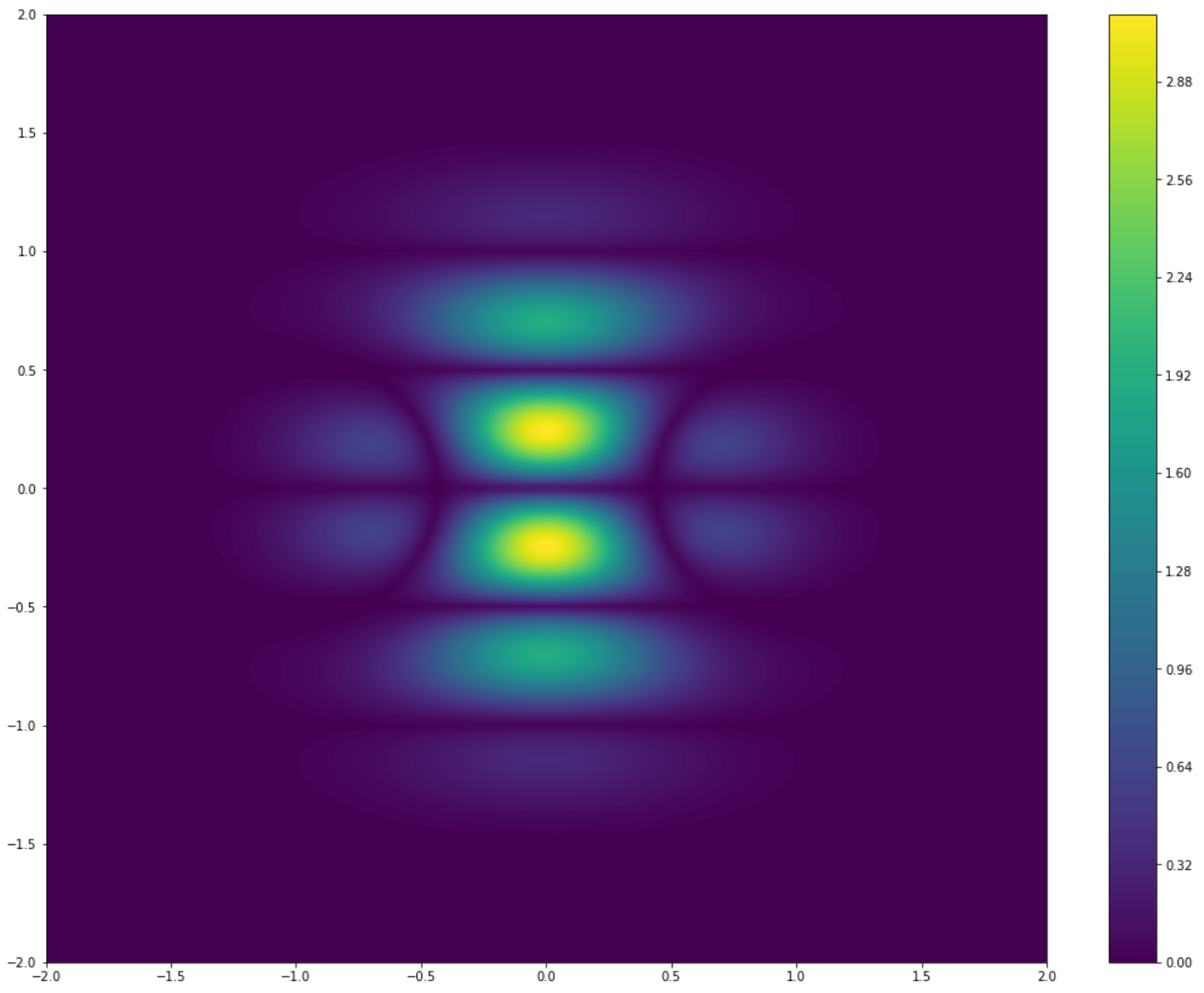}
  \caption{}
  \label{fig:sub1}
\end{subfigure}%
\begin{subfigure}{.5\textwidth}
  \centering
  \includegraphics[width=0.9\linewidth]{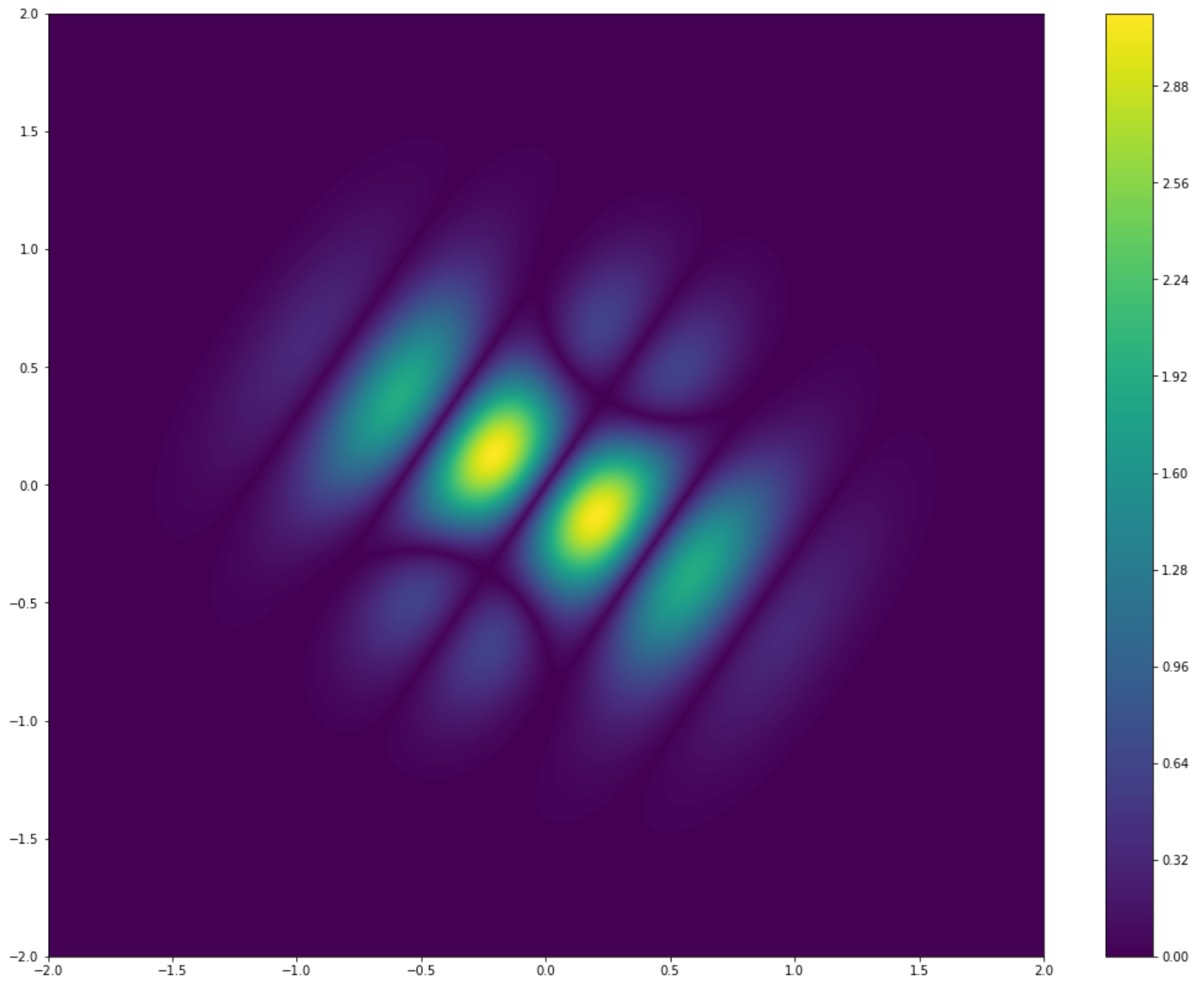}
  \caption{}
  \label{fig:sub2}
\end{subfigure}
\caption{Visualisation of the conclusion of Theorem \ref{theorem:parallel_lines}: the window $g$ is chosen to be the first Hermite basis function, $g=H_1$. The function $f \in \mathcal{V}_\s^\theta(\mathcal{R} g)$ is chosen to have defining sequence $\{ c_k \} \in \ell^2_\mathcal{O}(\Z)$ with $c_{-1}=1,c_1=i$ and $c_k=0$ for $k \in \Z \setminus \{-1,1\}$. The constant $\s$ is set to one. The above plots are contour plots of the function $Q \coloneqq ||V_gf|^2-|V_g(f_\times)|^2|$ on the rectangle $[-2,2]\times[-2,2]$. We observe, that $Q$ vanishes on parallel lines of the form $R_\theta(\R \times \tfrac{1}{2}\Z) \supset R_\theta(\R \times \Z)$. In plot (A) we set $\theta=0$ and in plot (B) we set $\theta=1$. Note that plot (B) is simply a rotation of (A) since $H_1$ is an eigenfunction of the fractional Fourier transform.}
\label{fig:hermite}
\end{figure}

\begin{figure}
\centering
\begin{subfigure}{.5\textwidth}
  \centering
  \includegraphics[width=0.9\linewidth]{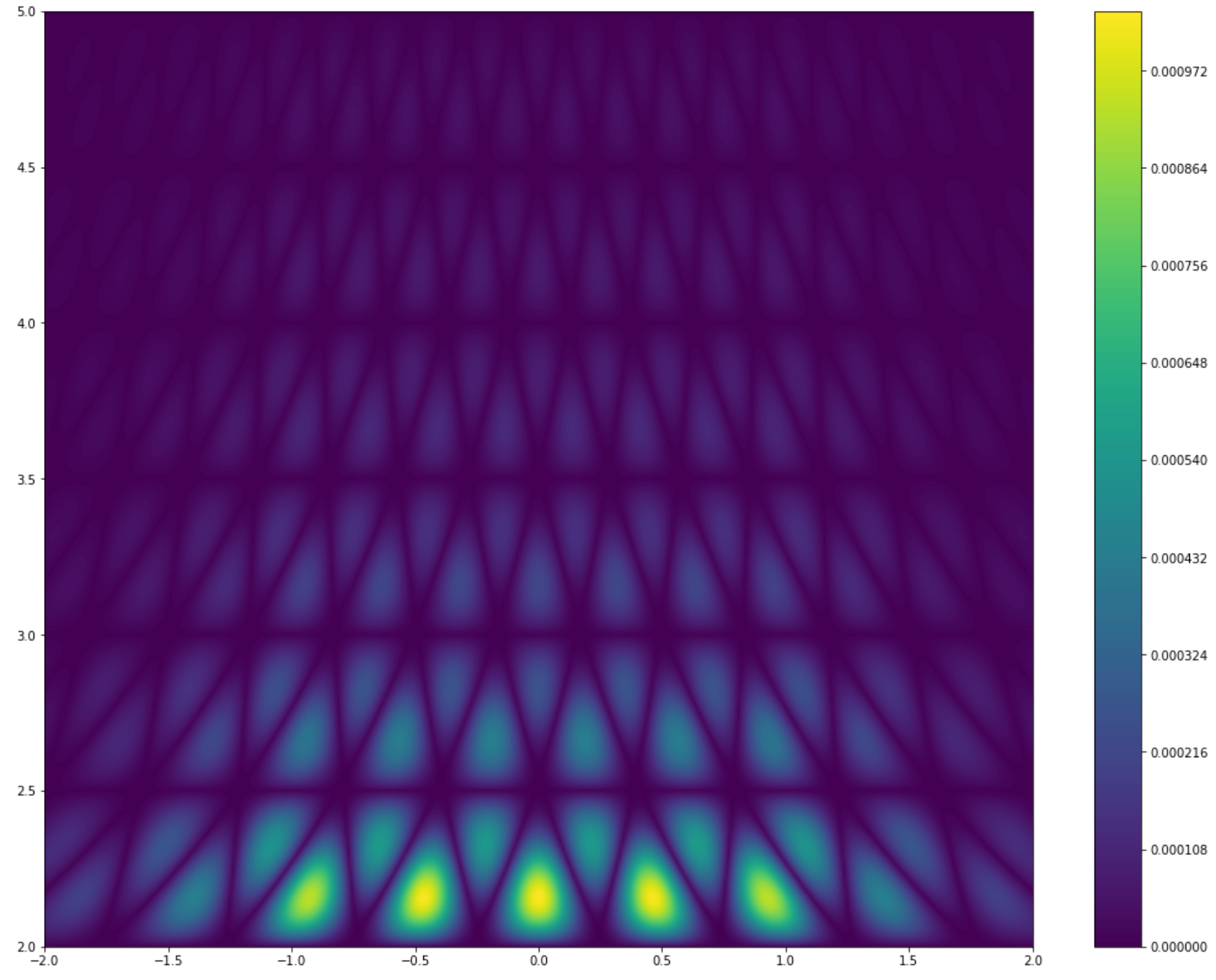}
  \caption{}
  \label{fig:sub3}
\end{subfigure}%
\begin{subfigure}{.5\textwidth}
  \centering
  \includegraphics[width=0.9\linewidth]{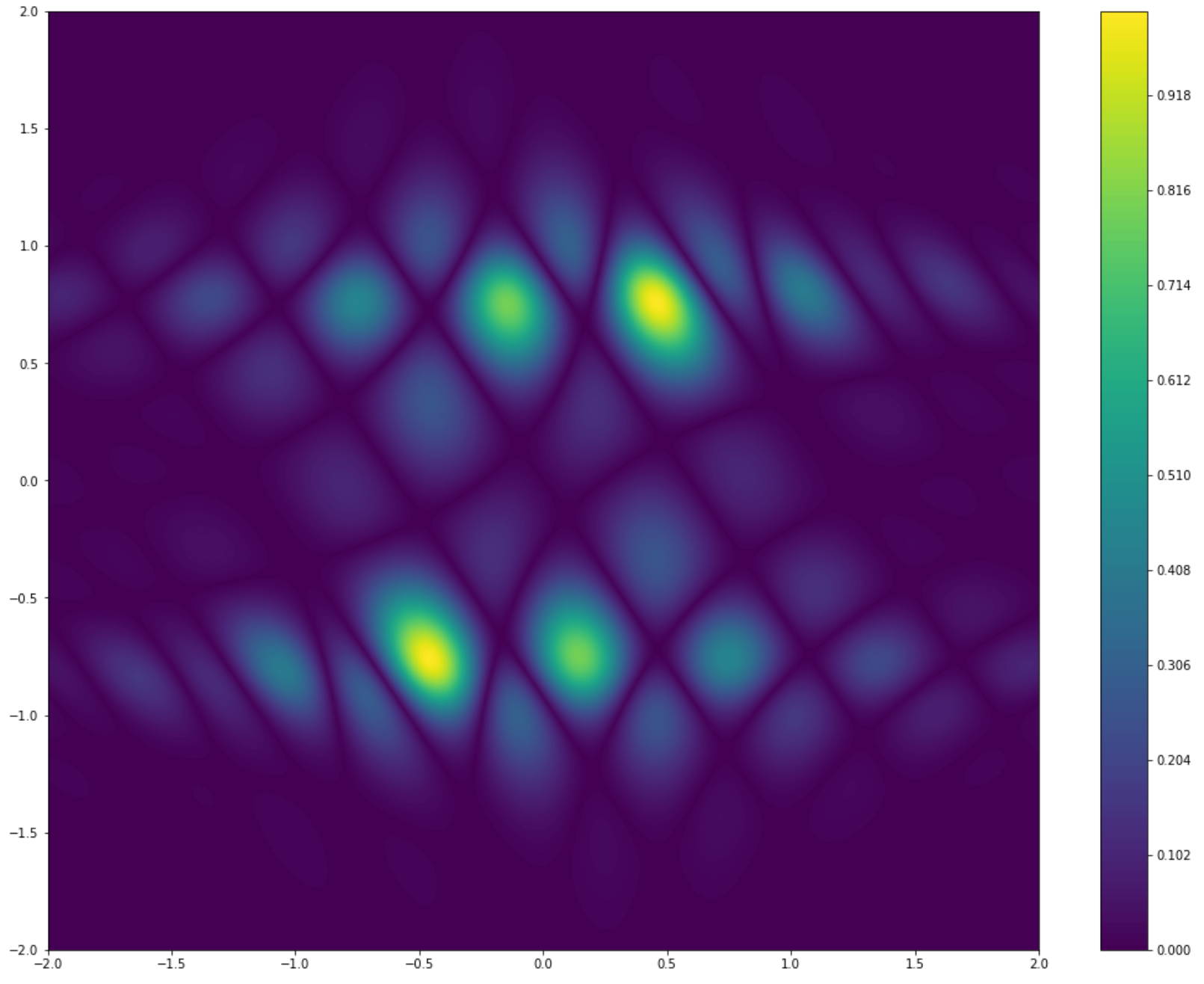}
  \caption{}
  \label{fig:sub4}
\end{subfigure}
\caption{Visualisation of the conclusion of Theorem \ref{theorem:parallel_lines} with window function $g(t)\coloneqq e^{-|t|}+\frac{1}{1+t^2}$. The function $f \in \mathcal{V}_\s^\theta(\mathcal{R} g)$ is chosen to have defining sequence $\{ c_k \} \in \ell^2_\mathcal{O}(\Z)$ with $c_{-1}=1,c_1=i$ and $c_k=0$ for $k \in \Z \setminus \{-1,1\}$. The constant $\s$ is set to one. The above plots are contour plots of the function $Q \coloneqq ||V_gf|^2-|V_g(f_\times)|^2|$ on the rectangle $[-2,2]\times[2,5]$ resp. $[-2,2]\times[-2,2]$. We observe, that $Q$ vanishes on parallel lines of the form $R_\theta(\R \times \tfrac{1}{2}\Z) \supset R_\theta(\R \times \Z)$. In plot (A) we set $\theta=0$ and in plot (B) we set $\theta=-1$.}
\label{fig:exp_arctan}
\end{figure}

\begin{proof}
(1) Assume by contradiction that $f \sim f_\times$ and let $\nu \in \T$ such that $f = \nu f_\times$. Further, let $J \coloneqq \{ k \in \Z : c_k \neq 0 \}$. Since $\{ c_k \} \in c_{00}(\Z)$, the set $J$ is finite and we have
$$
\sum_{k \in J} (c_k - \nu \overline{c_k}) T_{\s k}^\theta u = \sum_{k \in J} (c_k - \nu \overline{c_k}) \ft_{-\theta} T_{\s k} \ft_\theta u = 0.
$$
The invertibility of the fractional Fourier transform implies that
$$
\sum_{k \in J} (c_k - \nu \overline{c_k}) T_{\s k}(\ft_\theta u) = 0.
$$
Since $0 \neq u \in \lt$ it holds that $\ft_\theta u \neq 0$. Therefore, the system of translates $\{ T_{\s k}(\ft_\theta u) : k \in \Z \}$ is a linearly independent system in the vector space $\lt$ \cite[Proposition 9.6.2]{christensenBook}. From the additional property that $|J|<\infty$, it follows that
\begin{equation}\label{eq:ck}
    c_k - \nu \overline{c_k}=0 \ \ \ \forall k \in J.
\end{equation}
Equation \eqref{eq:ck} implies that the map $C : J \to \C$, defined by $C(k)=c_k$, is equivalent to its complex conjugate, $C \sim \overline{C}$. By Lemma \ref{lemma:funcion_values_lines}, the values $\{ c_k \}$ lie on a line in the complex plane passing through the origin, contradicting the assumption that $\{ c_k \} \in \ell^2_\mathcal{O}(\Z)$.

(2) Since $\{ T_{\s k}^\theta u : k \in \Z \}$ is a Bessel sequence and $\ell_\mathcal{O}^2(\Z) \subset \ell^2(\Z)$, the series $f = \sum_{k \in \Z} c_k T_{\s k}^\theta u$ and $f_\times = \sum_{k \in \Z} \overline{c_k} T_{\s k}^\theta u$ converge unconditionally by Corollary \ref{cor:uncond_convergence}.
In a similar fashion as in the first part of the present proof, we assume by contradiction that $f \sim f_\times$. This implies that there exists a $\nu \in \T$ such that
$$
\sum_{k \in \Z} (c_k - \nu \overline{c_k}) T_{\s k}^\theta u = 0.
$$
By assumption, the system $\{ T_{\s k}^\theta u : k \in \Z \}$ is $\omega$-independent. Hence, $c_k = \nu \overline{c_k}$ for every $k \in \Z$ which shows that the points $\{ c_k \}$ lie on a line in the complex plane passing through the origin. This is a contradiction to the assumption that $\{ c_k \} \in \ell^2_\mathcal{O}(\Z)$.
\end{proof}

\begin{remark}[$\omega$-independent Bessel sequences: necessary and sufficient conditions]
Let $u\in \lt$. Since $\ft_\theta$ is a unitary operator, it follows from the definition of the fractional Fourier shift that the system $\{ T_{\s k}^\theta u : k \in \Z \}$ is an $\omega$-independent Bessel sequence if and only if the system $\{ T_{\s k} h : k \in \Z \}$ is an $\omega$-independent Bessel sequence where $h \coloneqq \ft_\theta u$. Assume for the sake of simplicity that $\s = 1$. Then $\{ T_k h : k \in \Z \}$ is a Bessel sequence if and only if the \emph{periodization} $\Phi_h$ of $|\ft h|^2$, defined by
$$
\Phi_h(t) = \sum_{k \in \Z} |\ft h (t + k)|^2,
$$
is an element of $L^\infty[0,1]$ \cite[Theorem 10.19]{basisTheoryPrimer}. Assume, in addition, that $h \in W_0$ where $W_0$ denotes the Wiener amalgam space,
$$
W_0 = \left \{ h \in C(\R) : \sum_{k \in \Z} \max_{t \in [k,k+1]} |h(t)| < \infty \right \}.
$$
Then $\{ T_k h : k \in \Z \}$ is $\ell^\infty$-independent, i.e. $\sum_{k \in \Z} c_k T_k h \neq 0$ for every $\{ c_k \} \in \ell^\infty(\Z) \setminus \{ 0 \}$, if and only if the periodization $\Phi_h$ does not vanish on $[0,1]$ \cite[Theorem 2.1]{grst}. This holds, in particular, under the assumptions of the previous theorem where $\{ c_k \} \in \ell^2(\Z) \subset \ell^\infty(\Z)$.
\end{remark}

\subsection{Consequences for STFT phase retrieval}

At this juncture, we are prepared to transfer the foregoing considerations to the uniqueness problem arising in STFT phase retrieval. Recall that $(g,\mathcal{L})$ is said to be a uniqueness pair of the STFT phase retrieval problem with window function $g \in \lt$ and sampling set $\mathcal{L} \subseteq \R^2$ if every $f \in \lt$ is determined up to a global phase by $|V_gf(\mathcal{L})|$ (see Definition \ref{definition:uniqueness_pair}). Call a set $\mathcal{P} \subset \R^2$ a set of shifted parallel lines if
$$
\mathcal{P} = z + R_\theta (\R \times h \Z)
$$
for some $z \in \R^2, h>0$ and some $\theta \in \R$ ($R_\theta$ denotes the rotation matrix defined in equation \eqref{def:rotation_matrix}). Combining the statements derived in Section \ref{subsec:general_case} yields the following result.

\begin{theorem}\label{thm:main_parallel_lines}
Let $g \in \lt$ be a window function and $\mathcal{P} \subset \R^2$. Then $(g,\mathcal{P})$ is never a uniqueness pair of the STFT phase retrieval problem, provided that $\mathcal{P}$ is a set of shifted parallel lines.
\end{theorem}
\begin{proof}
The case $g = 0$ is trivial. Therefore, assume that $g$ does not vanish identically. Since $\mathcal{P}$ is a set of shifted parallel lines, there exist constants $\theta \in \R$ and $\s >0$ as well as a vector $z=(a,b) \in \R^2$ such that
$$
\mathcal{P} = z+ R_\theta (\R \times \tfrac{1}{\s}\Z).
$$

\textbf{Case 1: $z=0$.} Choose a function $f \in \mathcal{V}_\s^\theta(\mathcal{R}g)$ with defining sequence belonging to the (non-empty) intersection $\ell^2_\mathcal{O}(\Z) \cap c_{00}(\Z)$. By Theorem \ref{theorem:parallel_lines}, we have $|V_gf(R_\theta (\R \times \tfrac{1}{\s}\Z))| = |V_g(f_\times)(R_\theta (\R \times \tfrac{1}{\s}\Z))|$ whereas $f \nsim f_\times$ by Theorem \ref{theorem:non_uniqueness}. This yields the assertion for $z=0$.

\textbf{Case 2: $z \neq 0$.} If $p = (x,\omega) \in \R^2$ then the covariance property of the STFT (Lemma \ref{lemma:comm_relations}(6)) shows that for every $f \in \lt$ one has
$$
V_gf(z+p) = V_gf(a+x,b+\omega) = e^{-2\pi i a \omega} V_g(T_{-a}M_{-b}f)(p).
$$
Taking absolute values, it follows that if $h_1,h_2 \in \lt$ are two functions which produce identical spectrogram values on the parallel lines $R_\theta (\R \times \tfrac{1}{\s}\Z)$ then $f_1 \coloneqq M_bT_a h_1, f_2 \coloneqq M_b T_a h_2$ produce identical spectrogram values on the shifted parallel lines $\mathcal{P}=z+R_\theta (\R \times \tfrac{1}{\s}\Z)$. Moreover, one has $h_1 \sim h_2$ if and only if $f_1 \sim f_2$ (the equivalence relation $\sim$ is invariant under time-frequency shifts). Therefore, if $h \in \mathcal{V}_\s^\theta(\mathcal{R}g)$ has defining sequence in the intersection $\ell^2_\mathcal{O}(\Z) \cap c_{00}(\Z)$, then the spectrograms of $f_1 = M_bT_a h$ and $f_2 = M_bT_a(h_\times)$ agree on $\mathcal{P}$ whereas $f_1 \nsim f_2$.
\end{proof}

Note that if the window function $g$ is assumed to be the centered Gaussian $g(t)=e^{-\pi t^2}$ then it was shown in \cite{alaifari2020phase} that $(g,\mathcal{P})$ is not a uniqueness pair provided that $\mathcal{P}$ is a set of shifted parallel lines. In comparison, Theorem \ref{thm:main_parallel_lines} above makes no assumption on the window function $g$. In analogy to the definition of a set of shifted parallel lines we say that $\mathcal{S} \subset \R^2$ is a shifted lattice if $\mathcal{S}$ arises from an ordinary lattice $\mathcal{L}=L\Z^2, L \in \mathrm{GL}_2(\R),$ via a translation by a vector $z \in \R^2$, i.e.
$$
\mathcal{S} = z + \mathcal{L}.
$$
As a consequence of Theorem \ref{thm:main_parallel_lines} we obtain the statement that lattices never achieve uniqueness no matter how the window function $g$ is chosen.

\begin{theorem}\label{thm:main_lattice}
Let $g \in \lt$ be a window function and $\mathcal{S} \subset \R^2$. Then $(g,\mathcal{S})$ is never a uniqueness pair of the STFT phase retrieval problem, provided that $\mathcal{S}$ is a shifted lattice.
\end{theorem}
\begin{proof}
Since $\mathcal{S}$ is a shifted lattice, there exist constants $\theta \in \R$ and $\s >0$ as well as a vector $z \in \R^2$ such that
$$
\mathcal{S} \subset z + R_\theta (\R \times \tfrac{1}{\s}\Z).
$$
Thus, the assertion follows from Theorem \ref{thm:main_parallel_lines}.
\end{proof}

\begin{remark}
If $\mathcal{S} = z + L\Z^2, z \in \R^2,$ with a non-invertible matrix $L \in \R^{2 \times 2}$ then the column vectors of $L$ are linearly dependent which implies that $\mathcal{S}$ is contained in a single line. Clearly, the conclusions made in Theorem \ref{thm:main_lattice} are still valid in this setting.
\end{remark}

\begin{remark}
A subset $\mathcal{C}\subseteq X$ of a (real or complex) vector space $X$ is said to be a cone if $\kappa \mathcal{C} \subseteq \mathcal{C}$ for every $\kappa > 0$. A cone is called infinite-dimensional if it is not contained in any finite-dimensional subspace of $X$. Observe that the intersection $\ell^2_\mathcal{O}(\Z) \cap c_{00}(\Z)$ is a cone in the vector space of complex sequences and this cone is infinite-dimensional. Thus, for every $\s>0, \theta \in \R$ and $z=(a,b) \in \R^2$ the set
$$
\mathcal{K}_\s^\theta(z,g) \coloneqq \left \{ M_bT_a \sum_{k \in \Z} c_k T_{\s k}^\theta g : \{ c_k \} \in \ell^2_\mathcal{O}(\Z) \cap c_{00}(\Z)  \right \}
$$
is an infinite-dimensional cone in $\lt$. In an abstract language, Theorem \ref{thm:main_parallel_lines} and Theorem \ref{thm:main_lattice} imply that for every window function $g \in \lt$ and every shifted lattice $\mathcal{S}$ (resp. set of shifted parallel lines $\mathcal{P}$), one can associate an infinite-dimensional cone of the form $\mathcal{K}_\s^\theta(z,g)$ with the following property: if $f_1 \in \mathcal{K}_\s^\theta(z,g)$ then there is an $f_2 \in \mathcal{K}_\s^\theta(z,g)$ such that
\begin{enumerate}
    \item $|V_gf_1(\mathcal{S})| = |V_gf_2(\mathcal{S})|$ (resp. $|V_gf_1(\mathcal{P})| = |V_gf_2(\mathcal{P})|$)
    \item $f_1 \nsim f_2$.
\end{enumerate}
\end{remark}

\subsection{Real-valuedness}

As outlined in the introduction, real-valued window functions are prominent choices in many applications. Note that the previously derived function pairs $f,f_\times$ which implied the non-uniqueness of the STFT phase retrieval problem from lattice samples were, in general, complex-valued (even if the window function is real-valued). We now examine the question of whether a prior real-valuedness assumption on the underlying signal space achieves uniqueness from lattice samples, i.e. one specializes the considered input functions to belong to the space
$
L^2(\R,\R) = \left \{ f \in \lt : f \ \text{is real-valued} \right \}.
$
To that end, we first introduce the subspace of Hermitian sequences.

\begin{definition}
The subspace $\ell^2_\mathcal{H}(\Z)$ of Hermitian sequences is defined by
$$
\ell^2_\mathcal{H}(\Z) \coloneqq \left \{ \{ c_k \} \in \ell^2(\Z) : c_k=\overline{c_{-k}} \right \}.
$$
\end{definition}

As for the set $\ell^2_\mathcal{O}(\Z)$, the space $\ell^2_\mathcal{H}(\Z)$ is invariant under complex conjugation.
Functions $f \in \mathcal{V}_\s(u) = \mathcal{V}_\s^0(u)$, with defining sequence belonging to $\ell^2_\mathcal{H}(\Z)$, have the following property.

\begin{proposition}\label{prop:r_v}
Let $u \in \lt$ and suppose that $h \in \mathcal{V}_\s(u)$. If $h$ has defining sequence $\{ c_k \} \in \ell^2_\mathcal{H}(\Z) \cap c_{00}(\Z)$ then $\ft h$ is real-valued provided that $\ft u$ is real-valued. If, in addition, the system of translates $\{ T_{\s k} u : k\in \Z \}$ forms a Bessel sequence, then the same conclusion holds true, under the weaker assumption that $\{ c_k \} \in \ell^2_\mathcal{H}(\Z)$.
\end{proposition}
\begin{proof}
Suppose that $h = \sum_{k \in \Z} c_k T_{\s k} u$ is such that $\{ c_k \} \in \ell^2_\mathcal{H}(\Z) \cap c_{00}(\Z)$. Then
\begin{equation*}
    \begin{split}
        \ft h = \sum_{k \in \Z} c_k M_{-\s k}\ft u = c_0 \ft u + \sum_{k=1}^\infty c_k M_{-\s k} \ft u + \sum_{k=1}^\infty c_{-k} M_{\s k} \ft u
    \end{split}
\end{equation*}
and since $\ft u$ is real-valued and $\{ c_k \} \in \ell^2_\mathcal{H}(\Z)$ we have $c_0 = \overline{c_0}$ which implies that $c_0 \in \R$. Moreover,
$$
\sum_{k=1}^\infty c_{-k} M_{\s k} \ft u = \overline{\sum_{k=1}^\infty c_{k} M_{-\s k} \ft u},
$$
thereby proving the first part of the statement. To prove the second part, assume that $\{ T_{\s k} u : k\in \Z \}$ forms a Bessel sequence. By Corollary \ref{cor:uncond_convergence} the series $\sum_{k \in \Z} c_k T_{\s k} u$ converges unconditionally. Invoking the continuity and linearity of the Fourier transform yields the statement with the same argument as above.
\end{proof}

As an application of Proposition \ref{prop:r_v}, we construct non-equivalent, real-valued functions for which their spectrograms agree on certain lattices in the time-frequency plane.

\begin{theorem}\label{thm:main_R_valued}
Let $g \in L^2(\R,\R)$ be a real-valued window function and let $\mathcal{L}=L\Z^2 \subset \R^2$ be a lattice with generating matrix
\begin{equation*}
    L = \begin{pmatrix}
a & b\\
c & d
\end{pmatrix} \in \mathrm{GL}_2(\R) .
\end{equation*}
If $a,b$ are linearly dependent over $\Q$ and $c,d$ are linearly dependent over $\Q$ then there exist two real-valued functions $f_1,f_2 \in L^2(\R,\R)$ with the property $|V_gf_1(\mathcal{L})| = |V_gf_2(\mathcal{L})|$ but $f_1 \nsim f_2$.
\end{theorem}
\begin{proof}
The case $g = 0$ is trivial. Therefore, assume in the following that $g \neq 0$ and let $\theta \in \R$.

\textbf{Case 1: rectangular lattices.} First assume that the lattice $\mathcal{L}$ is rectangular, i.e. $\mathcal{L}=\alpha \Z \times \beta \Z$ for some $\alpha, \beta \in \R \setminus \{ 0 \}$. In this case, $\mathcal{L}$ is generated by the diagonal matrix $L = \mathrm{diag}(\alpha,\beta)$ which clearly satisfies the assumptions of the Theorem. Noting that the intersection $\ell^2_\mathcal{O}(\Z) \cap \ell^2_\mathcal{H}(\Z) \cap c_{00}(\Z)$ is non-empty, we select a function $h \in \mathcal{V}_\s^\theta(\mathcal{R}g)$ with defining sequence $\{ c_k \} \in \ell^2_\mathcal{O}(\Z) \cap \ell^2_\mathcal{H}(\Z) \cap c_{00}(\Z)$ and set
$$
f_1 = h, \ \ f_2 = h_\times.
$$
For $\theta = \frac{\pi}{2}$ we have 
\begin{equation}\label{f11}
    f_1 = \mathcal{R} \left ( \ft \sum_{k \in \Z} c_k T_{\s k}\ft \mathcal{R} g \right),
\end{equation}
and $u \coloneqq \ft \mathcal{R}g$ has a real-valued Fourier transform. Identity \eqref{f11} together with Proposition \ref{prop:r_v} shows that $f_1$ is real-valued. Since $\ell^2_\mathcal{O}(\Z), \ell^2_\mathcal{H}(\Z)$ and $c_{00}(\Z)$ are all invariant under complex conjugation we immediately conclude that $f_2$ is real-valued as well. Theorem \ref{theorem:non_uniqueness} shows that $f_1$ and $f_2$ do not agree up to a global phase. But for $\s = \frac{1}{\alpha}$ we have
$$
R_\theta(\R \times \tfrac{1}{\s}\Z) = R_{\frac{\pi}{2}}(\R \times \tfrac{1}{\s}\Z) = \tfrac{1}{\s}\Z \times \R \supset \alpha \Z \times \beta \Z,
$$
whence $|V_gf_1(\mathcal{L})|=|V_gf_2(\mathcal{L})|$ by Theorem \ref{theorem:parallel_lines}, thereby proving the statement for rectangular lattices.

\textbf{Case 2: generalization.}
Suppose that $\mathcal{L}=L\Z^2$ satisfies the assumptions above, i.e. the row elements of $L$ are linearly dependent over $\Q$. Hence, there exist $p,q \in \Q$ such that $a=pb$ and $c=qd$. Write $p \coloneqq \frac{p_1}{p_2} \in \Q$ with  $p_1 \in \Z, p_2 \in \Z \setminus \{ 0 \}$ and $q \coloneqq \frac{q_1}{q_2} \in \Q$ with $q_1 \in \Z, q_2 \in \Z \setminus \{ 0 \}$. Then for $v=(v_1,v_2) \in \Z^2$ we have
\begingroup
\renewcommand*{\arraystretch}{1.5}
\begin{equation*}
    Lv = \begin{pmatrix}
pbv_1 + b v_2\\
qdv_1 + dv_2
\end{pmatrix}
=
\underbrace{\begin{pmatrix}
\frac{b}{p_2} & 0\\
0 & \frac{d}{q_2}
\end{pmatrix}}_{\coloneqq L'}
\begin{pmatrix}
w_1\\
w_2
\end{pmatrix}
\end{equation*}
\endgroup
with $w_1 =p_1 v_1 + p_2 v_2 \in \Z$ and $w_2 = q_1 v_1 + q_2 v_2 \in \Z$. This shows that $\mathcal{L}$ is contained in a rectangular lattice generated by $L'$. Therefore, Case 1 applies and the existence of two real-valued, non-equivalent functions $f_1,f_2 \in L^2(\R,\R)$, for which their spectrograms agree on $\mathcal{L}$, follows.
\end{proof}

\section{Conclusions and open problems}\label{sec:conclusion}

In the concluding section of the paper, we shall compare our results to previous work and address open problems suggested by the statements made in Section \ref{sec:main}.

\subsection{Gaussian windows}

The authors of \cite{alaifari2020phase} show that in the case of the Gaussian window $g(t) = \varphi(t)=e^{-\pi t^2}$, there exists no lattice of the form $\mathcal{L}=\alpha \Z \times \beta \Z$ such that every $f\in\lt$ is determined up to a unimodular constant by $|V_gf(\mathcal{L})|$. In their article, the provided functions $f_1,f_2 \in \lt$ which produce the same spectrogram values on the parallel lines $\R \times \beta \Z$ but are not equivalent, $f_1 \nsim f_2$, are of the form
\begin{equation*}
    \begin{split}
        f_1(t) &= e^{-\pi t^2} \left ( \cosh(\tfrac{\pi t}{\beta}) +i \sinh (\tfrac{\pi t}{\beta}) \right ) \\
        f_2(t) &= e^{-\pi t^2} \left ( \cosh(\tfrac{\pi t}{\beta}) -i \sinh (\tfrac{\pi t}{\beta}) \right ).
    \end{split}
\end{equation*}
A simple calculation shows that $f_1,f_2 \in \mathcal{V}_{\frac{1}{2\beta}}(\varphi)$. Moreover, $f_1$ has defining sequence $\{ c_k \} \in c_{00}(\Z)$ given by
\begin{equation*}
    c_k = \begin{cases}
    \frac{1-i}{2}e^{\frac{\pi}{4b^2}}, & k=-1 \\
    \frac{1+i}{2}e^{\frac{\pi}{4b^2}}, & k=1 \\
    0 , & \text{else}
    \end{cases}
\end{equation*}
and the defining sequence of $f_2$ is given by $\{ \overline{c_k} \}$. In particular, one has $\{ c_k \} \in \ell^2_\mathcal{O}(\Z) \cap \ell^2_\mathcal{H}(\Z) \cap c_{00}(\Z)$ and $f_2 = (f_1)_\times$. This shows that $f_1$ and $f_2$ are special cases of the functions derived in Section \ref{sec:main}.

\subsection{Restriction to proper subspaces and comparison to previous work}

Theorem \ref{thm:main_lattice} states that no matter how the window function is chosen, a discretization of the STFT phase retrieval problem from lattice samples or parallel lines forces one to restrict the input space of functions to a proper subspace or subset $S \subsetneq \lt$. This was successfully done in previous work: in \cite{grohsliehr1} it was shown that if the window function is a Gaussian and the underlying signal space satisfies a mild support condition, then sampling on a lattice implies uniqueness. Additionally, in the same article, it was shown that if the signal space is a Gaussian shift-invariant space with an irrational step-size, sampling on a lattice again suffices to achieve uniqueness. In \cite{grohsliehr2} it was further proved that without an assumption on the step-size of the Gaussian shift-invariant space, the uniqueness property holds, assuming knowledge of the spectrogram on parallel lines in the time-frequency plane. Finally, the article \cite{ALAIFARI202134} establishes uniqueness results for the class of real-valued, band-limited functions from samples in a 1-dimensional lattice and window functions which satisfy a non-vanishing condition in Fourier space.

\subsection{Open problems}

In this article, we have established foundational barriers in the discretization of the STFT phase retrieval problem from lattice samples. Theorem \ref{thm:main_lattice} showed that uniqueness from (shifted) lattices is never guaranteed for the signal space $\lt$. A natural question arising at this point is the inquiry if Theorem \ref{thm:main_R_valued} also holds in this generality. For instance, a hexagonal lattice is not covered by the assumptions of Theorem \ref{thm:main_R_valued}.

\begin{problem}
Suppose that $g \in L^2(\R,\R)$ is a real-valued window function. Does there exist a lattice $\mathcal{L} \subset \R^2$ such that every real-valued map $f \in L^2(\R, \R)$ is determined up to a sign by $|V_gf(\mathcal{L})|$?
\end{problem}

A lattice is a special case of a separated (or: uniformly discrete) set. Recall that this is a set $\mathcal{L} \subset \R^2$ with the property
$$
\inf_{\substack{\ell,\ell' \in \mathcal{L} \\ \ell \neq \ell'}} |\ell - \ell'| > 0,
$$
where $|\ell - \ell'|$ denotes the Euclidean distance of $\ell$ and $\ell'$.
A fruitful question for future research suggested by the results reported here is the following.

\begin{problem}\label{problem1}
Does there exist a window function $g \in \lt$ and a separated set $\mathcal{L} \subset \R^2$ such that $(g,\mathcal{L})$ is a uniqueness pair for the STFT phase retrieval problem?
\end{problem}

A more narrow form of Problem \ref{problem1} is the question of whether or not adjoining a finite set of points $\{ p_1, \dots, p_N\} \subset \R^2$ to a lattice $\mathcal{L}$ produces a sampling set with the uniqueness property. Such questions arise in the completeness problem of complex exponentials \cite{Boivin, Levinson, REDHEFFER19771} which is equivalent to the uniqueness problem in Paley-Wiener spaces. Recall that a set $\Lambda \subset \R^d, d \in \N$, is called a uniqueness set for a subset $Q \subset C(\R^d)$ of the space of continuous functions on $\R^d$ if the implication
$$
f(\lambda) = h(\lambda) \ \forall \lambda \in \Lambda \implies f(t)=h(t) \ \forall  \in t \in \R^d
$$
holds true for all $f,h \in Q$. The notion of exactness, excess, and deficiency concerns economy \cite[Chapter 3.1]{Young}.

\begin{definition}\label{definition:deficiency}
Let $Q \subset C(\R^d), d \in \N$, and $\Lambda \subset \R^d$. Then $\Lambda$ is said to be exact in $Q$ if it is a uniqueness set for $Q$ but fails to be a uniqueness set for $Q$ on the removal of any one term of $\Lambda$. If $\Lambda$ becomes exact when $N$ terms are removed, then we say that it has excess $N$; if it becomes exact when $N$ terms $\{ p_1, \dots, p_N \} \subset \R^d$ are adjoint, then $\Lambda$ is said to have deficiency $N$. If it never becomes exact, then we say that $\Lambda$ has deficiency $N=\infty$.
\end{definition}

Now let $g \in \lt$ be a window function. The notions provided in Definition \ref{definition:deficiency} directly transfer to phase retrieval: we say that $\Lambda \subset \R^2$ is a uniqueness set for the STFT phase retrieval problem with deficiency $N \in \N_0 \cup \{ \infty \}$ if there exists a set $\{ p_1, \dots, p_N \} \subset \R^2$ (the case $N=\infty$ is included) such that every $f \in \lt$ is determined up to a global phase by $|V_gf(\Lambda \cup \{ p_1, \dots, p_N \})|$. According to Theorem \ref{thm:main_lattice}, every lattice has deficiency $N>0$.

\begin{problem}\label{problem2}
Does there exists a window functions $g \in \lt$ and a lattice $\mathcal{L} \subset \R^2$ with deficiency $N<\infty$? 
\end{problem}

Note that a positive answer to Problem \ref{problem2} provides a positive answer to Problem \ref{problem1}.

\vspace{0.5cm}

\textbf{Acknowledgements.}
The authors highly appreciate the valuable and helpful comments made by the reviewers.

\bibliographystyle{abbrv}
\bibliography{bibfile}

\end{document}